\documentclass[letterpaper]{article}

\usepackage{amsmath}
\usepackage{amsfonts}
\usepackage{amsthm}
\usepackage[letterpaper]{geometry}
\usepackage{hyperref}
\usepackage[]{ytableau}
\usepackage[all, arc, curve, frame]{xy}

\numberwithin{equation}{section}

\title{When Does the Set of $(a, b, c)$-Core Partitions Have a Unique Maximal Element?}
\author{Amol Aggarwal} 

\begin{document}

\maketitle

\begin{abstract}
In 2007, Olsson and Stanton gave an explicit form for the largest $(a, b)$-core partition, for any relatively prime positive integers $a$ and $b$, and asked whether there exists an $(a, b)$-core that contains all other $(a, b)$-cores as subpartitions; this question was answered in the affirmative first by Vandehey and later by Fayers independently. In this paper we investigate a generalization of this question, which was originally posed by Fayers: for what triples of positive integers $(a, b, c)$ does there exist an $(a, b, c)$-core that contains all other $(a, b, c)$-cores as subpartitions? We completely answer this question when $a$, $b$, and $c$ are pairwise relatively prime; we then use this to generalize the result of Olsson and Stanton. 
\end{abstract}

\newtheorem{thm}{Theorem}[section]
\newtheorem{prop}[thm]{Proposition}
\newtheorem{lem}[thm]{Lemma}
\newtheorem{cor}[thm]{Corollary}
\newtheorem{que}[thm]{Question}

\section{Introduction}

A {\itshape partition} is a finite, nonincreasing sequence $\lambda = (\lambda_1, \lambda_2, \ldots , \lambda_r)$ of positive integers. The sum $\sum_{i = 1}^r \lambda_i$ is the {\itshape size} of $\lambda$ and is denoted by $|\lambda|$; the integer $r$ is the {\itshape length} of $\lambda$. A partition $\mu = (\mu_1, \mu_2, \ldots , \mu_s)$ is a {\itshape subpartition} of $\lambda$ if $s \le r$ and $\mu_i \le \lambda_i$ for each integer $i \in [1, s]$; in this case, we say that $\mu \subseteq \lambda$.  

We may represent $\lambda$ by a {\itshape Young diagram}, which is a collection of $r$ left-justified rows of cells with $\lambda_i$ cells in row $i$. The {\itshape hook length} of any cell $C$ in the Young diagram is defined to be the number of cells to the right of, below, or equal to $C$. For instance, Figure 1 shows the Young diagrams and hook lengths of the partitions $(6, 4, 2, 2, 1, 1)$ and $(5, 3, 1, 1)$. Let $\beta (\lambda)$ denote the set of hook lengths in the leftmost column of the Young tableaux associated with $\lambda$; equivalently, $\beta (\lambda) = (\lambda_1 + r - 1, \lambda_2 + r - 2, \ldots , \lambda_r)$. For instance, Figure 1 shows that $\beta (6, 4, 2, 2, 1, 1) = \{ 11, 8, 5, 4, 2, 1 \}$ and $\beta (5, 3, 1, 1) = (8, 5, 2, 1)$. 

For any set of positive integers $A = \{ a_1, a_2, \ldots , a_k \}$, a partition is an {\itshape $A$-core} if no cell of its Young diagram has hook length in $A$. Let the set of $A$-cores be $C(A)$; Figure 1 shows that $(6, 4, 2, 2, 1, 1) \in C(3, 7)$ and $(5, 3, 1, 1) \in C(3, 7, 11)$. 

Core partitions are known to be related to representations of the symmetric group; for instance, Olsson and Stanton use simultaneous core partitions in [11] to prove the Navarro-Willems conjecture for symmetric groups. Core partitions are also known to be related to the alcove geometry for certain types of Coxeter groups (see [4, 8, 9]). Recently, there has been a growing interest in simultaneous core partitions because of their relationship with numerical semigroups (see [1, 2, 4, 14, 16]). 

During the past decade, combinatorialists have studied properties of $C(A)$ when $|A| = 2$ (see [1, 3, 4, 5, 6, 8, 9, 10, 11, 14, 15]). For instance, Anderson showed that $|C(a, b)| = \binom{a + b}{a} / (a + b)$ if $a$ and $b$ are relatively prime; in particular, there are finitely many $(a, b)$-cores [3]. 

This implies that there is an $(a, b)$-core of maximum size. Auckerman, Kane, and Sze conjectured in [5] that this size is $(a^2 - 1)(b^2 - 1) / 24$. This was verified in 2007 by Olsson and Stanton, who also found the core of this size explicitly in terms of $a$ and $b$ [11]. Specifically, they established the following result. 

\begin{thm}
\label{largestcore2}

For any relatively prime positive integers $a$ and $b$, there is a unique $(a, b)$-core $\kappa_{a, b}$ of maximum size; a positive integer is in $\beta (\kappa_{a, b})$ if and only if it is of the form $ab - ia - jb$ for some positive integers $i$ and $j$. 
\end{thm}

Figure 1 depicts the Young diagram of $\kappa_{3, 7} = (6, 4, 2, 2, 1, 1)$. In their proof of \hyperref[largestcore2]{Theorem \ref*{largestcore2}}, Stanton and Olsson showed that $\kappa_{n, n + 1}$ contains every other $(n, n + 1)$-core as a subpartition, for each integer $n \ge 2$ [11]. They then asked whether $\kappa_{a, b}$ contains all other $(a, b)$-cores as subpartitions, for every pair of relatively prime positive integers $(a, b)$. Vandehey answered this question in the affirmative in 2009 through the use of abacus diagrams [15]. Recently, Fayers obtained the same result by analyzing actions of the affine symmetric group on the set of $a$-cores (and on the set of $b$-cores) [8]. 

To see an example of Vandehey's theorem, let $A = (3, 5)$. The nonempty partitions in $C(A)$ are $\{ (1), (2), (1, 1), (3, 1), (2, 1, 1), (4, 2, 1, 1) \}$, and every element of $C(A)$ is contained in the $(3, 5)$-core $(4, 2, 1, 1)$. However, this containment phenomenon does not necessarily hold when $|A| \ge 3$ and $\gcd A = 1$. For instance, if $A = \{ 3, 4 , 5 \}$, then the nonempty elements of $C(A)$ are $(1, 1)$ and $(2)$; neither of these is contained in the other. 

\begin{figure}[t]
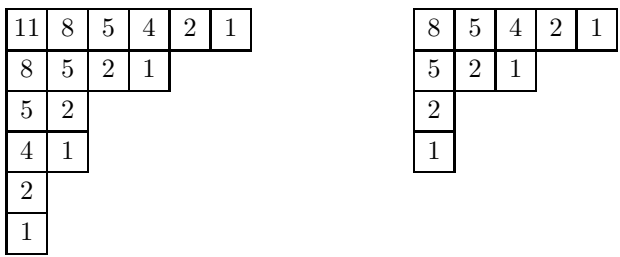

\begin{center}
\ytableausetup{centertableaux}
\begin{ytableau}
11 & 8 & 5 & 4 & 2 & 1 & \none & \none & \none & \none & 8 & 5 & 4 & 2 & 1 \\
8 & 5 & 2 & 1 & \none & \none & \none & \none & \none & \none & 5 & 2 & 1 \\
5 & 2 & \none & \none & \none & \none & \none & \none & \none & \none & 2 \\
4 & 1 & \none & \none & \none & \none & \none & \none & \none & \none & 1 \\
2 \\
1 
\end{ytableau}

\caption{To the left is the Young diagram of $\kappa_{3, 7} = (6, 4, 2, 2, 1, 1)$ and to the right is the Young diagram of $(5, 3, 1, 1)$; each cell contains its hook length. } 
\end{center}
\end{figure}

For any set of positive integers $A$, we say that $C(A)$ has a {\itshape unique maximal element} if there is an $A$-core $\kappa_A$ that contains every other $A$-core as a subpartition; in this case, the set $A$ is said to be {\itshape UM}. In [7], Fayers asked the following question. 

\begin{que}
\label{umquestion}
What triples of positive integers $(a, b, c)$ are UM?
\end{que}

Vandehey's result implies a partial result in this direction. For any set of positive integers $A$, let $S(A)$ be the numerical semigroup generated by $A$; equivalently, $A$ consists of all linear combinations of elements in $A$ with nonnegative integer coefficients. Due to the known fact that an $(a, b)$-core is an $(a + b)$-core (see [2], for instance), Vandehey's result implies that $(a, b, c)$ is UM if $a$ and $b$ are relatively prime and $c \in S(a, b)$. Recently, Yang, Zhong, and Zhou showed that $(2k + 1, 2k + 2, 2k + 3)$ is not UM for any positive integer $k$ [16]. 

In this paper we give a partial answer to \hyperref[umquestion]{Question \ref*{umquestion}}. We call a triple of positive integers $(a, b, c)$ {\itshape aprimitive} if either $a \in S(b, c)$, $b \in S(a, c)$, or $c \in S(a, b)$. The following theorem gives a restriction on triples that can be UM. 

\begin{thm}
\label{umtriples}
Suppose that $(a, b, c)$ is a triple of positive integers such that $\gcd (a, b, c) = 1$; let $p = \gcd (a, b)$, $q = \gcd (a, c)$, $r = \gcd (b, c)$, and $d$, $e$, $f$ be integers such that $(a, b, c) = (dpq, epr, fqr)$ as ordered triples. If $(a, b, c)$ is UM, then $(d, e, f)$ is aprimitive. 
\end{thm}

Not all triples of the form given by the above theorem are UM; for instance, we will see in Section 2 that $(4, 5, 6)$ is not UM. However, we may use \hyperref[umtriples]{Theorem \ref*{umtriples}} to answer \hyperref[umquestion]{Question \ref*{umquestion}} completely when $a$, $b$, and $c$ are pairwise relatively prime. 

\begin{cor}
\label{relativelyprime}
If $a < b < c$ are pairwise relatively prime positive integers, then $(a, b, c)$ is UM if and only if $c \in S(a, b)$. 
\end{cor}

\begin{proof}
As noted previously, Vandehey's theorem implies that $(a, b, c)$ is UM if $c \in S(a, b)$. Setting $p = q = r = 1$ in \hyperref[umtriples]{Theorem \ref*{umtriples}} yields the converse. 
\end{proof}

\hyperref[relativelyprime]{Corollary \ref*{relativelyprime}} can be viewed as a converse to Vandehey's theorem; it also generalizes the previously mentioned result of Yang, Zhong, and Zhou. 

Using \hyperref[umtriples]{Theorem \ref*{umtriples}}, we will also be able to express the unique maximal $(a, b, c)$-core $\kappa_{a, b, c}$ in terms of $a$, $b$, and $c$ if the triple $(a, b, c)$ is UM. 

\begin{thm}
\label{maximalcore}
Suppose that $A = (a, b, c)$ is a triple of positive integers that is UM; let $p = \gcd (a, b)$, $q = \gcd (a, c)$, $r = \gcd (b, c)$, and $d, e, f$ be integers such that $(a, b, c) = (dpq, epr, fqr)$ as ordered triples. If $f \in S(d, e)$, then a positive integer is in $\beta (\kappa_A)$ if and only if it is of the form $(de + f) pqr - ia - jb - kc$ for some positive integers $i$, $j$, and $k$. 
\end{thm}

Observe that letting $p = q = r = 1$ in \hyperref[maximalcore]{Theorem \ref*{maximalcore}} yields \hyperref[largestcore2]{Theorem \ref*{largestcore2}} of Olsson and Stanton, due to Vandehey's theorem. 

The proofs of \hyperref[umtriples]{Theorem \ref*{umtriples}} and \hyperref[maximalcore]{Theorem \ref*{maximalcore}} use a recently developed characterization of simultaneous cores using numerical semigroups, which we will explain further in Section 2. 

\section{Proofs of Theorems 1.3 and 1.5}
In this section, we first explain a bijection (originally due to Stanley and Zanello in [14] when $|A| = 2$ and later generalized to arbitrary sets $A$ by Amdeberhan and Leven in [2]) between $A$-cores and order ideals of some poset $P(A)$. We will then use this bijection to obtain a preliminary necessary condition for a set $A$ to be UM. Using this condition, we will establish \hyperref[umtriples]{Theorem \ref*{umtriples}} and \hyperref[maximalcore]{Theorem \ref*{maximalcore}}. 

Now let us define the poset $P(A)$. The elements of $P(A)$ are those of $\mathbb{Z}_{\ge 0} \backslash S(A)$, the set of positive integers not contained in the numerical semigroup generated by $A$. Notice that if $\gcd A = 1$, then $|P(A)| < \infty$; we will suppose that this is the case for the remainder of the section. The order on $P(A)$ is fixed by requiring $p \in P(A)$ to be greater than $q \in P(A)$ if $p - q \in S(A)$. Under this partial order, $P(A)$ is a poset; we will follow the poset terminology given in Chapter 3 of Stanley's text [12, 13]. Figure 2 depicts the Hasse diagrams of the posets $P(3, 7)$ and $P(3, 7, 11)$. 

The following lemma is due to Amdeberhan and Leven in [1]. 

\begin{lem}
\label{bijection}
There is a bijection between $C(A)$ and the set of order ideals of $P(A)$. Specifically, for each partition $\lambda$, the set $\beta (\lambda)$ is an order ideal of $P(A)$ if and only if $\lambda$ is an $A$-core. 
\end{lem}

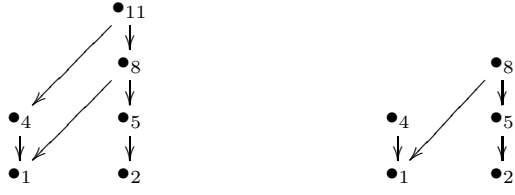
\begin{figure}[t]
\begin{center}
\xymatrix @R=.75pc {& & & & \bullet_{11} \ar[d] \ar[ldd] \\
& & & & \bullet_8 \ar[d] \ar[ldd] & & & & \bullet_8 \ar[d] \ar[ldd] \\
& & & \bullet_4 \ar[d] & \bullet_5 \ar[d] & & & \bullet_4 \ar[d] & \bullet_5 \ar[d] \\
& & & \bullet_1 & \bullet_2 & & & \bullet_1 & \bullet_2} 
\caption{The Hasse diagrams of $P(3, 7)$ and $P(3, 7, 11)$ are shown on the left and right, respectively.} 
\end{center} 
\end{figure} 

For instance, suppose that $A \subseteq \{ 3, 7, 11 \}$; then, $(5, 3, 1, 1)$ is an $A$-core and its beta set $\{ 8, 5, 2, 1 \}$ is an order ideal of $P(A)$. Furthermore, $\kappa_{3, 7} = (6, 4, 2, 2, 1, 1)$ is a $(3, 7)$-core and its beta set $\{ 1, 2, 4, 5, 8, 11 \}$ is an order ideal of $P(3, 7)$; however, $\kappa_{3, 7}$ is not a $(3, 7, 11)$-core and its beta set is not an order ideal of $P(3, 7, 11)$. 

From \hyperref[bijection]{Lemma \ref*{bijection}}, there is an $A$-core $\kappa_A'$ such that $\beta (\kappa_A') = P(A)$. The following result states that $\kappa_A'$ is the unique maximal element of $C(A)$ if $A$ is UM. 

\begin{cor}
\label{uniquemaximalelement}
If a set of positive integers $A$ is UM, then $\kappa_A' = \kappa_A$. 
\end{cor}

\begin{proof}
The bijection in \hyperref[bijection]{Lemma \ref*{bijection}} is length preserving; since the longest order ideal of $P(A)$ is $P(A)$, the longest $A$-core is $\kappa_A'$. Therefore, $\kappa_A'$ is not contained in any other $A$-core, which implies that $\kappa_A$ is the unique maximal element of $C(A)$ because $A$ is UM. Thus, $\kappa_A' = \kappa_A$. 
\end{proof} 

Now, we call a poset $P$ {\itshape poset-UM} if $P$ contains a unique maximal element. For instance, Figure 2 shows that $P(3, 7)$ is poset-UM with unique maximal element $11$ and that $P(3, 7, 11)$ is not poset-UM since it has both $4$ and $8$ as maximal elements. The following known lemma gives examples of posets that are poset-UM. 

\begin{lem}
\label{twoposet}
Suppose that $a$ and $b$ are relatively prime positive integers; then, $P(a, b)$ is poset-UM with maximal element $ab - a - b$. Equivalently, a positive integer is in $P(a, b)$ if and only if it is of the form $ab - ja - hb$ for some integers $j \in [1, b - 1]$ and $h \in [1, a - 1]$. 
\end{lem}

\noindent The following proposition yields a preliminary necessary condition for a set of positive integers to be UM. 

\begin{prop}
\label{posetum}
If a set of positive integers $A$ is UM, then $P(A)$ is poset-UM. 
\end{prop}

\begin{proof}
Suppose that $P(A)$ is not poset-UM but that $A$ is UM. Consider the element $m \in P(A)$ of maximum magnitude; for instance, if $A = (3, 7, 11)$, then $m = 8$. Since $P(A)$ is not poset-UM, there is an order ideal $I \subseteq P(A)$ containing $m$ but not equal to $P(A)$. By \hyperref[bijection]{Lemma \ref*{bijection}}, there are $A$-cores $\lambda = (\lambda_1, \lambda_2, \ldots , \lambda_r)$ and $\kappa = \kappa_A = (\kappa_1, \kappa_2, \ldots , \kappa_s)$ such that $\beta (\lambda) = I$ and $\beta (\kappa) = P(A)$. By \hyperref[uniquemaximalelement]{Corollary \ref*{uniquemaximalelement}}, $\kappa$ is the unique maximal element of $C(A)$; therefore, $\lambda \subset \kappa$. Hence, $m - r + 1 = \lambda_1 \le \kappa_1 = m - s + 1$, which is a contradiction since $s = |P(A)| > |I| = r$. 
\end{proof}

As an application, the above proposition implies that $(3, 7, 11)$ is not UM. Observe that the converse of \hyperref[posetum]{Proposition \ref*{posetum}} does not always hold. For instance, suppose that $A = \{ 4, 5, 6 \}$; then, $P(A) = \{ 1, 2, 3, 7 \}$ is poset-UM with unique maximal element $7$. Therefore the longest $A$-core is $\kappa_A = (3, 1, 1, 1)$, which does not contain the $A$-core $(2, 2)$ as a subpartition. 

We will now classify all triples of positive integers $(a, b, c)$ whose associated posets $P(a, b, c)$ are poset-UM. The following proposition forms a bijection between the maximal elements of $P(a, b, c)$ and the maximal elements of $P(d, e, f)$ for particular triples $(a, b, c)$ and $(d, e, f)$. 

\begin{prop}
\label{maximalelements}
Suppose that $(a, b, c)$ is a triple of positive integers with $\gcd (a, b, c) = 1$; let $\gcd (a, b) = p$, $\gcd (a, c) = q$, $\gcd (b, c) = r$, and $d, e, f$ be integers such that $(a, b, c) = (dpq, epr, fqr)$ as ordered triples. For $m$ and $n$ positive integers, $(m - 1) d + (n - 1) e - f$ is a maximal element of $P(d, e, f)$ if and only if $(mr - 1) dpq + (nq - 1) epr - fqr$ is a maximal element of $P(a, b, c)$. 
\end{prop}

\begin{proof}
Let $s = (m - 1) d + (n - 1) e - f$ and $t = (mr - 1) dpq + (nq - 1) epr - fqr$. Suppose that $s$ is a maximal element of $P(d, e, f)$; we will show that $t$ is a maximal element of $P(a, b, c)$. Let us first verify that $t \in P(a, b, c)$. Suppose to the contrary that $t \in S(a, b, c)$. Then, there are nonnegative integers $h, i, j$ such that $(mr - 1) dpq + (nq - 1) epr - fqr = hdpq + iepr + jfqr$. Since $\gcd (a, b, c) = 1$, we have that $\gcd (r, p) = \gcd (r, q) = 1 = \gcd (r, d)$. The previous equality implies that $(mr - 1 - h) dpq$ is a multiple of $r$, which yields $h = rh' - 1$ for some positive integer $h'$. Similarly, $i = qi' - 1$ for some positive integer $i'$; therefore, $(m - h') dp + (n - i') ep = (j + 1) f$. Thus, $j = j' p - 1$ for some positive integer $j'$; hence $s = (m - 1) d + (n - 1) e - f = (h' - 1) d + (i' - 1) e + (j' - 1) f \in S(d, e, f)$, which is a contradiction. 

Now, to see that $t$ is maximal, it suffices to verify that $t + a, t + b, t + c \in S(a, b, c)$. Observe that $t + a = mdpqr + (nq - 1) epr - fqr = pqr (s + d) + epr (q - 1) + fqr (p - 1) \in S(dpq, epr, fqr)$ because $s + d \in S(d, e, f)$ by the maximality of $s$. By similar reasoning, $t + b \in S(a, b, c)$; since $t + c = (mr - 1) dpq + (nq - 1) epr \in S(a, b, c)$, it follows that $t$ is a maximal element of $P(a, b, c)$. 

This implies that $t$ is a maximal element of $P(a, b, c)$ if $s$ is a maximal element of $P(d, e, f)$. Through similar reasoning, one may show that $s$ is a maximal element of $P(d, e, f)$ if $t$ is a maximal element of $P(a, b, c)$. 
\end{proof}

\noindent The following corollary reduces the classification of triples $(a, b, c)$ whose associated posets $P(a, b, c)$ are poset-UM to the case when $a$, $b$, and $c$ are pairwise relatively prime. 

\begin{cor}
\label{notrelativelyprime}
Suppose that $(a, b, c)$ is a triple of positive integers such that $\gcd (a, b, c) = 1$; let $\gcd (a, b) = p$, $\gcd (a, c) = q$, $\gcd (b, c) = r$, and $d, e, f$ be integers such that $(a, b, c) = (dpq, epr, fqr)$ as ordered triples. Then, $P(a, b, c)$ is poset-UM if and only if $P(d, e, f)$ is poset-UM. 
\end{cor}

\begin{proof}
Suppose that $P(d, e, f)$ is not poset-UM and let $s_1$ and $s_2$ be two distinct maximal elements of $P(d, e, f)$. Since $s_1$ and $s_2$ are maximal, we have that $s_1 + f, s_2 + f \in S(d, e)$; thus, there are positive integers $m_1, n_1, m_2, n_2$ such that $s_1 = (m_1 - 1) d + (n_1 - 1) e - f$ and $s_2 = (m_2 - 1) d + (n_2 - 1) e - f$. Let $t_1 = (r m_1 - 1) dpq + (q n_1 - 1) epr - fqr$ and $t_2 = (r m_2 - 1) dpq + (q n_2 - 1) epr - fqr$; by \hyperref[maximalelements]{Proposition \ref*{maximalelements}}, $t_1$ and $t_2$ are maximal elements of $P(a, b, c)$. Since $s_1$ and $s_2$ are distinct, $t_1 \ne t_2$; thus $P(a, b, c)$ has two distinct maximal elements and is therefore not poset-UM. 

Now, suppose that $P(a, b, c)$ is not poset-UM and let $t_1$ and $t_2$ be two distinct maximal elements of $P(a, b, c)$. As above, there are positive integers $m_1', n_1', m_2', n_2'$ such that $t_1 = (m_1' - 1) dpq + (n_1' - 1) epr - fqr$ and $t_2 = (m_2' - 1) dpq + (n_2' - 1) epr - fqr$. We claim that $m_1'$ is a multiple of $r$. Indeed, since $t_1$ is maximal, $m_1' dpq + (n_1' - 1) epr - fqr = t_1 + a \in S(a, b, c)$; therefore, there are nonnegative integers $h, i, j$ such that $m_1' dpq + (n_1' - 1) epr - fqr = hdpq + iepr + jfqr$. Since $t_1 \notin S(a, b, c)$, we have that $h = 0$. Therefore, $r$ divides $m_1' dpq$; the fact that $\gcd (d, r) = \gcd (p, r) = \gcd (q, r) = 1$ thus yields $r$ divides $m_1'$. Hence, there is an integer $m_1$ such that $m_1' = m_1 r$; by similar reasoning, there are integers $n_1, m_2, n_2$ such that $n_1' = n_1 q$, $m_2' = m_2 r$, and $n_2' = n_2 q$. By \hyperref[maximalelements]{Proposition \ref*{maximalelements}}, $s_1 = (m_1 - 1) d + (n_1 - 1) e - f$ and $s_2 = (m_2 - 1) d + (n_2 - 1) e - f$ are unique maximal elements of $P(d, e, f)$. Since $t_1$ and $t_2$ are distinct, $s_1 \ne s_2$; this implies that $P(d, e, f)$ has two maximal elements and is therefore not poset-UM. 
\end{proof}

\noindent Now we classify all triples of pairwise relatively prime positive integers $(a, b, c)$ such that $P(a, b, c)$ is poset-UM. 

\begin{prop}
\label{umrelativelyprime}
If $a < b < c$ are pairwise relatively prime positive integers, then $P(a, b, c)$ is poset-UM if and only if $c \in S(a, b)$. 
\end{prop}

\begin{proof} 
If $c \in S(a, b)$, then $P(a, b, c) = P(a, b)$ because $S(a, b, c) = S(a, b)$; therefore, the proposition follows from \hyperref[twoposet]{Lemma \ref*{twoposet}}. Now suppose that $c \notin S(a, b)$; we will show that $P(a, b, c)$ has at least two distinct maximal elements. By \hyperref[twoposet]{Lemma \ref*{twoposet}}, there are positive integers $s_1 \in [1, b - 1]$ and $t_1 \in [1, a - 1]$ such that $c = ab - s_1 a - t_1 b$. Let $k$ be the largest positive integer such that $ic \notin S(a, b)$ for each integer $i \in [1, k]$. By \hyperref[twoposet]{Lemma \ref*{twoposet}}, there are integers $s_1, s_2, \ldots , s_k \in [1, b - 1]$ and $t_1, t_2, \ldots , t_k \in [1, a]$ such that $ic = ab - s_i a - t_i b$ for each integer $i \in [1, k]$; moreover, let $(s_0, t_0) = (0, a)$ and $(s_{k + 1}, t_{k + 1}) = (b, 0)$. Observe that $t_i < t_j$ if and only if $s_i > s_j$ for each $i, j \in [0, k + 1]$. Indeed, otherwise there would exist some $i > j$ such that $t_i < t_j$ and $s_i < s_j$; this would imply that $(i - j) c = (s_j - s_i) a + (t_j - t_i) b \in S(a, b)$, which is a contradiction since $i - j \in [1, k]$. 

Let $m, n \in [1, k]$ be integers such that $s_m = \min_{i \in [1, k]} s_i$ and $t_n = \min_{i \in [1, k]} t_i$. Observe by the above that $s_n = \max_{i \in [1, k]} s_i$ and $t_m = \max_{i \in [1, k]} t_i$. We claim that $k = m + n - 1$. We will first show that $k < m + n - 1$; suppose otherwise, so in particular $(m + n) c = 2ab - (s_m + s_n) a - (t_m + t_n) b \notin S(a, b)$. Thus $2ab - (s_m + s_n) a - (t_m + t_n) b = ab - s_{m + n} a -  t_{m + n} b$, so $ab = (s_m + s_n - s_{m + n}) a + (t_m + t_n - t_{m + n}) b$. Since $|s_m + s_n - s_{m + n}| < 2b$, $|t_m + t_n - t_{m + n}| < 2a$, and $\gcd (a, b) = 1$, this implies that either $s_m + s_n - s_{m + n} = b$ and $t_m + t_n = t_{m + n}$ or $s_m + s_n = s_{m + n}$ and $t_m + t_n - t_{m + n} = a$. Without loss of generality, suppose that the former holds; then $s_{m + n} = s_m + s_n - b < s_m$, which contradicts the minimality of $s_m$. Therefore, $k \le m + n - 1$. 

To see that $k \ge m + n - 1$, observe that $(m + n - i) c = ab - (s_m + s_n - s_i) a - (t_m + t_n - t_i) b$ for each integer $i \in [1, k]$. Since $0 < s_m \le s_i \le s_n$ and $0 < t_n \le t_i \le t_m$, we obtain that $(m + n - i) c$ is of the form $ab - ah - bj$ for some positive integers $h$ and $j$; hence \hyperref[twoposet]{Lemma \ref*{twoposet}} implies that $ic \notin S(a, b)$ for each integer $i \in [1, m + n - 1]$, which yields $k = m + n - 1$. 

Now let $\{ r_0, r_1, \ldots , r_{k + 1} \}$ be a permutation of $\{ 0, 1, 2, \ldots , k, k + 1 \}$ such that $s_{r_0} < s_{r_1} < \cdots < s_{r_{k + 1}}$. Observe that $t_{r_k} < t_{r_{k - 1}} < \cdots < t_{r_1}$; $r_0 = 0$; $r_{k + 1} = k + 1$; $r_1 = m$; and $r_k = n$. Moreover, let $p_i = ab - (s_{r_i} + 1) a - (t_{r_{i + 1}} + 1) b$ for each integer $i \in [0, k]$. Observe that $p_i \in P(a, b, c)$ for each integer $i \in [0, k]$; indeed, suppose to the contrary that there is some $p_i \in S(a, b, c)$. Then, there would exist integers $f \in [0, b]$, $h \in [0, a]$, and $j \in [0, k]$ such that $ab - (s_{r_i} + 1)a - (t_{r_{i + 1}} + 1)b = ab - (s_j - f) a - (t_j - h) b$. Thus $s_j > s_{r_i}$ and $t_j > t_{r_{i + 1}} > t_{r_i}$, which is a contradiction. Furthermore, observe that $p_i + a \notin P(a, b, c)$ for each integer $i \in [0, k]$ since $p_i + a = ab - s_{r_i} a - (t_{r_{i + 1}} + 1) b = r_i c + (t_{r_i} - t_{r_{i + 1}} - 1) b$ and $t_{r_i} > t_{r_{i + 1}}$; by similar reasoning, $p_i + b \notin P(a, b, c)$. 

Let $k' \in [1, k]$ be the integer such that $r_{k'} = k$. Then $(k - r_{k' - 1}) c = ab - (s_{r_{k'}} - s_{r_{k' - 1}}) a - (a + t_{r_{k'}} - t_{r_{k' - 1}}) b$, so the minimality of $s_{r_m}$ implies that $r_{k' - 1} = k - m$; similarly, $r_{k' + 1} = k - n$. We claim that $p_{k' - 1} + c \in S(a, b, c)$. Supposing otherwise, \hyperref[twoposet]{Lemma \ref*{twoposet}} implies that there exist positive integers $j$ and $h$ such that $2ab - (s_{k - m} + s_1 + 1) a - (t_k + t_1 + 1)b = p_{k'} + c = ab - ja - hb$; therefore, $(s_{k - m} + s_1 + 1 - j) a + (t_k + t_1 + 1 - h) b = ab$. Since $|s_{k - m} + s_1 + 1 - j| < 2b$, $|t_k + t_1 + 1 - h| < 2a$, and $\gcd (a, b) = 1$, we have that either $s_{k - m} + s_1 \ge b$ or $t_k + t_1 \ge a$.

Since $(k - m) c = (n - 1) c = ab - (s_n - s_1) a - (a + t_n - t_1) b$, we have that $s_{k - m} + s_1 = s_n < b$. This implies that $t_k + t_1 \ge a$. Since $k = m + n - 1$, we have that $s_k = s_m + s_n - s_1$ and $t_k = t_m + t_n - t_1$. Moreover, since $ab - (s_m + s_n) a - (t_m + t_n - a) b = (m + n) c \in S(a, b)$, we have that $t_k + t_1 = t_m + t_n \le a$ by \hyperref[twoposet]{Lemma \ref*{twoposet}}. Thus, $t_m + t_n = a$ and hence $a(b - s_m - s_n) = (m + n) c$. Since $a$ and $c$ are relatively prime, $b - s_m - s_n$ is a multiple of $c$; this contradicts the fact that $c < b$. Therefore, $p_{k' - 1} + c \notin P(a, b, c)$. Since $p_{k' - 1} + a \notin P(a, b, c)$, $p_{k - 1} + b \notin P(a, b, c)$, and $p_{k' - 1} \in P(a, b, c)$, it follows that $p_{k' - 1}$ is a maximal element of $P(a, b, c)$. 

By similar reasoning, $p_{k'}$ is a maximal element of $P(a, b, c)$; since $p_{k'} \ne p_{k' - 1}$, this yields that $P(a, b, c)$ has two distinct maximal elements and is therefore not poset-UM. 
\end{proof}

\noindent We may now classify all triples of positive integers $(a, b, c)$ for which $P(a, b, c)$ is poset-UM. 

\begin{cor}
\label{classification}
Suppose that $(a, b, c)$ is a triple of positive integers such that $\gcd (a, b, c) = 1$; let $p = \gcd (a, b)$, $q = \gcd (a, c)$, $r = \gcd (b, c)$, and $d, e, f$ be integers such that $(a, b, c) = (dpq, epr, fqr)$ as ordered triples. Then, $P(a, b, c)$ is poset-UM if and only if $(d, e, f)$ is aprimitive. 
\end{cor}

\begin{proof}
This follows from \hyperref[notrelativelyprime]{Corollary \ref*{notrelativelyprime}} and \hyperref[umrelativelyprime]{Proposition \ref*{umrelativelyprime}}. 
\end{proof}

\noindent We may now establish \hyperref[umtriples]{Theorem \ref*{umtriples}}.  

\begin{proof}[Proof of Theorem 1.3]
This follows from \hyperref[posetum]{Proposition \ref*{posetum}} and \hyperref[classification]{Corollary \ref*{classification}}. 
\end{proof}

\noindent Using \hyperref[umtriples]{Theorem \ref*{umtriples}}, we may now establish \hyperref[maximalcore]{Theorem \ref*{maximalcore}}. 

\begin{proof}[Proof of Theorem 1.5] 
Let $f = md + ne$ for some nonnegative integers $m$ and $n$. By \hyperref[twoposet]{Lemma \ref*{twoposet}}, the unique maximal element of $P(d, e, f)$ is $de - d - e = (d + n - 1)e + (m - 1)d - f$. By \hyperref[posetum]{Proposition \ref*{posetum}}, $P(A)$ has a unique maximal element; by \hyperref[maximalelements]{Proposition \ref*{maximalelements}}, the unique maximal element of $P(A)$ is $(de + md + ne) pqr - dpq - epr - fqr = (de + f) pqr - a - b - c$. Hence, a positive integer is in $P(A)$ if and only if it is of the form $(de + f) pqr - ia - jb - kc$ for some positive integers $i$, $j$, and $k$. This implies the corollary since $\beta (\kappa_A) = P(A)$. 
\end{proof}

\section{Acknowledgements}
This research was conducted under the supervision of Joe Gallian at the University of Minnesota Duluth REU, funded by NSF Grant 1358659 and NSA Grant H98230-13-1-0273. The author heartily thanks Matt Fayers, Rishi Nath, and Joe Gallian for suggesting the topic of this project and for their valuable advice; the author also thanks David Moulton for his insightful discussions and Aaron Abrams and Richard Stanley for their suggestions.

\end{document}